\RequirePackage[l2tabu, orthodox]{nag}
\documentclass[a4paper,noamsfonts]{amsart}

\usepackage[T1]{fontenc}
\usepackage{fouriernc}
\usepackage[british]{babel}

\usepackage{amsmath}
\usepackage{amsthm}
\usepackage{amssymb}
\usepackage[foot]{amsaddr}
\usepackage{mathtools}
\mathtoolsset{mathic=true}

\usepackage{tikz}
\usetikzlibrary{cd}

\usepackage[dvipsnames]{xcolor}
\usepackage{graphicx}
\usepackage{pinlabel}

\usepackage{listings}
\usepackage{enumitem}
\usepackage[backend=biber, style=numeric, maxnames=8, isbn=false, sortcites=true, backref]{biblatex}
\AtEveryBibitem{\clearlist{language}}
\usepackage{hyperref}
\hypersetup{
    colorlinks,
    linkcolor={red!50!black},
    citecolor={blue!50!black},
    urlcolor={blue!80!black}
}

\usepackage{zref-clever}
\usepackage{microtype}

\zcsetup{noabbrev, cap}
\zcsetup{countertype={thm=theorem}}
\newtheorem{thm}{Theorem}
\numberwithin{thm}{section}
\newtheorem*{namedtheorem}{\theoremname}
\newcommand{\theoremname}{testing}
\newenvironment{named}[1]{\renewcommand{\theoremname}{#1}\begin{namedtheorem}}{\end{namedtheorem}}
\AddToHook{env/lem/begin}{\zcsetup{countertype={thm=lemma}}}
\newtheorem{lem}[thm]{Lemma}

\theoremstyle{definition}
\AddToHook{env/defn/begin}{\zcsetup{countertype={thm=definition}}}
\newtheorem{defn}[thm]{Definition}
\AddToHook{env/ex/begin}{\zcsetup{countertype={thm=example}}}
\newtheorem{ex}[thm]{Example}
\zcRefTypeSetup{cons}{
Name-sg = Construction ,
name-sg = construction ,
Name-pl = Constructions ,
name-pl = constructions ,
}
\AddToHook{env/cons/begin}{\zcsetup{countertype={thm=cons}}}
\newtheorem{cons}[thm]{Construction}
\newtheorem{prob}[thm]{Problem}
\AddToHook{env/prob/begin}{\zcsetup{countertype={thm=prob}}}
\zcRefTypeSetup{prob}{
Name-sg = Problem ,
name-sg = problem ,
Name-pl = Problems ,
name-pl = problems ,
}
\newtheorem*{idprob}{The Identification Problem}
\theoremstyle{remark}
\AddToHook{env/rem/begin}{\zcsetup{countertype={thm=remark}}}
\newtheorem{rem}[thm]{Remark}
\AddToHook{env/notation/begin}{\zcsetup{countertype={thm=notation}}}
\zcRefTypeSetup{notation}{
Name-sg = Notation ,
name-sg = notation ,
Name-pl = Notations ,
name-pl = notations ,
}
\newtheorem{notation}[thm]{Notation}

\newcommand{\df}{\textit}
\newcommand{\union}{\cup}

\newcommand{\mc}{\mathcal}

\newcommand{\mf}{\mathfrak}

\newcommand{\IN}{\mathbb{N}}
\newcommand{\IZ}{\mathbb{Z}}

\newcommand{\IC}{\mathbb{C}}
\newcommand{\IS}{\mathbb{S}}

\newcommand{\IP}{\mathbb{P}}
\newcommand{\IH}{\mathbb{H}}
\newcommand{\PSL}{\mathsf{PSL}}

\DeclareMathOperator{\tr}{tr}
\DeclareMathOperator{\Par}{Par}
\DeclareMathOperator{\ab}{\mathfrak{Ab}}

\DeclareMathOperator{\Aut}{Aut}

\begin{document}

\title[On rank two Kleinian groups with three parabolics]{On rank two Kleinian groups\\with three parabolics}
\author[A. Elzenaar]{Alex Elzenaar}
\address{School of Mathematics, Monash University, Melbourne}
\email{alexander.elzenaar@monash.edu}
\thanks{The author was supported by an Australian Government Research Training Program Scholarship during the period that this work was undertaken. This research was supported by Monash eResearch
capabilities, including the M3 HPC cluster. Part of this research was performed
while the author was visiting SLMath which is supported by NSF Grant No. DMS-2424139. He thanks Thomas Csizmadia, Connie On Yu Hui, Lavender Marshall, Gaven Martin, and Jessica Purcell
for helpful discussions. This document was prepared without the use of any generative AI}


\begin{abstract}
  The free group of rank $2$ is the fundamental group of the genus $2$ handlebody $\mc{H}$. We study discrete representations of this group into $ \PSL(2,\IC) $
  so that three disjoint simple closed curves on the conformal boundary $\partial_\infty \mc{H} $ are sent to parabolic elements. We show that the only infinite covolume groups of this
  form are maximal cusp groups on the boundary of genus $2$ Schottky space. We also exhibit hitherto unexpected finite covolume groups which do not arise
  from Heegaard splitting presentations of tunnel number $1$ links.
\end{abstract}

\maketitle

\section{Introduction}
The classification of discrete subgroups of $ \PSL(2,\IC) $ (called \df{Kleinian groups}) is an old problem dating back in its modern form to the 1970s~\cite{marden74}. In recent decades
most interest in this area has been directed towards the theory of ending laminations proposed by Thurston~\cite{thurston82}. The ending lamination theorem, proved by Brock,
Canary, and Minsky~\cite{brock12,minsky10} but building on work of many others, gives a complete set of topological invariants that can be used to classify Kleinian groups. It
is very useful for answering theoretical questions, but cannot easily be used to solve concrete classification problems. The prototypical such problem is:
\begin{idprob}
  Suppose $A_1,\ldots,A_k \in \PSL(2,\IC) $ are given such that $ G=\langle A_1,\ldots,A_k\rangle $ is discrete. Identify the
  hyperbolic isometry class of $ \IH^3/G $.
\end{idprob}
Checking that a group given in terms of generating matrices is discrete is also a very hard problem. Often with examples of matrix groups that arise in nature it is
possible to make an educated guess about discreteness by plotting the limit set, and then generally one will try to validate this guess by applying various methods
to solve the identification problem, or by attempting to find a $2$-generated indiscrete subgroup using J\o{}rgensen's inequality~\cite[Theorem~2.17]{matsuzaki_taniguchi}.

The most sophisticated techniques for attacking the identification problem for $ \PSL(2,\IC) $ exist in the special case that $ G = \langle A_1, A_2 \rangle $
where $ A_1 $ and $ A_2 $ are both parabolic matrices. Work in this area dates back to Riley's study of two-bridge link groups~\cite{riley72,riley75,riley79}.
The general picture was completed by Keen, Komori, and Series~\cite{keen94,komori98} who described the deformation space of free Kleinian groups generated by
two parabolic elements, and by Akiyoshi, Ohshika, Parker, Sakuma, and Yoshida~\cite{akiyoshi2020classification} who gave a complete description of all non-free Kleinian groups generated
by two parabolic elements. These results have been made computationally effective by recent work of Chesebro~\cite{chesebro25}
and Elzenaar, Martin, and Schillewaert~\cite{ems21} among others.

For more complicated groups there are only a few substantial computationally effective results. For instance, Riley considered semi-automated
discreteness testing for representations of more complicated link groups~\cite{riley82,riley83}. In addition there have been generalisations of Keen--Series theory to groups with
higher dimensional deformation spaces, like twice-punctured torus groups~\cite{series10}.

In this paper, we study the case of two generators $ X, Y \in \PSL(2,\IC) $. Since the free group on two generators is the fundamental group of the genus two
handlebody $ \mc{H} $, we can interpret words in $X$ and $Y$ as closed curves on the genus two surface $ \Sigma = \partial \mc{H} $. In the following definition,
recall that if $ f \in \PSL(2,\IC) $ satisfies $ \tr^2 f = 4 $ then either $ f $ is trivial or $ f $ is parabolic (has a unique fixed point on the Riemann sphere $ \IP^1 (\IC) $).
\begin{defn}
  Let $ \mc{X} $ be the character variety of representations $ F\{X,Y\} \to \PSL(2,\IC) $, where $ F \{X,Y\} $ is the
  free group on the symbols $ X $ and $ Y $. Suppose that $ \alpha $, $\beta$, and $\gamma $ are three distinct (up to homotopy)
  simple closed curves, all mutually disjoint, on $ \Sigma $; in other words, $ \{\alpha,\beta,\gamma\} $
  is a $2$-simplex in the curve complex of $ \Sigma $. Their \df{parabolic locus} is the set
  \begin{displaymath}
    \Par(\alpha,\beta,\gamma) = \{ \rho \in \mc{X} : \tr^2 \rho(\alpha) = \tr^2 \rho(\beta) = \tr^2 \rho(\gamma) = 4 \}.
  \end{displaymath}
\end{defn}

The fundamental problem with which we are concerned can now be stated. It is a special case of the identification problem for $ \PSL(2,\IC) $ where the matrix generators are drawn from
a subvariety of the matrix group.
\begin{prob}\label{prob:general}
  If $ \{\alpha,\beta,\gamma\} $ is a $2$-simplex in the curve complex of $ \Sigma $ so that none of the three curves are homotopically trivial in $ \mc{H} $,
  classify the discrete, non-elementary, non-Fuchsian representations in the parabolic locus $ \Par(\alpha,\beta,\gamma) $.
\end{prob}
In this note we discuss \zcref{prob:general} in the special case that all three curves are sent to parabolic elements rather than the identity.
Our work is directed by the interplay between tunnel number $1$ links and the system of equations defining the parabolic locus, so in \zcref{sec:motivation} we will
spend some time explaining the groups and other geometric objects that arise. We then summarise our main results in \zcref{sec:results}, which consist of (i) a complete
classification of all such groups which act discontinuously on an open subset of the Riemann sphere and (ii) an exhibition of examples of such groups which act ergodically on the sphere.
These examples show that some results from the two parabolic generators setting do not transfer cleanly to the general rank two case. In particular,
we previously believed that if $G$ is the image of a discrete representation where all three curves become parabolic then the domain of discontinuity of $G$ is non-empty
and $G$ is a maximal cusp group---this is true in particular when $ \alpha $ and $ \beta $ generate the group. We show that this is not true in general: if the domain of
discontinuity of $G$ is non-empty then $G$ is a maximal cusp (\zcref{thm:three_parabolics}), but there are discrete representations with \emph{lattice} images where the three curves become
parabolic. In fact, these lattices do not even arise naturally from tunnel number $1$ links. This shows that the structure theory for the parabolic locus must be more complicated than previously expected.

\subsection{Motivating constructions and examples}\label{sec:motivation}
There are two important constructions in Kleinian group theory and hyperbolic geometry which take a system of three disjoint curves $ \{\alpha,\beta,\gamma\} $ and
produce discrete, non-elementary, non-Fuchsian representations in $ \Par(\alpha,\beta,\gamma) $. These constructions furnish us with our prototypical examples of the groups which appear in our
study of \zcref{prob:general}.

\begin{cons}[Maximal cusp groups]
  Let $ \mc{S}_2 $ be genus two Schottky space, the space of holonomy representations of hyperbolic structures on a genus two handlebody $ \mc{H} $. It follows
  by classical results of Ahlfors, Bers, and Maskit that for every set of three disjoint simple closed $ \mc{H}$-incompressible curves $ \alpha, \beta, \gamma \subset \partial \mc{H} $
  there exists a representation $ \rho \in \Par(\alpha,\beta,\gamma) $ that lies on $ \partial \mc{S}_2 $ so that $ \rho(\alpha) $, $ \rho(\beta) $, and $ \rho(\gamma) $ are
  all parabolic~\cite[\S 4.3.2]{matsuzaki_taniguchi}. In fact, such representations are dense in the boundary~\cite{canary03}
  and are characterised by having circle-packing limit sets~\cite{keen91}.
\end{cons}

\begin{cons}[Holonomy groups of tunnel number one links]\label{cons:tunnum1}
  Let $ \mf{k} \subset \IS^3 $ be a tunnel number $1$ link~\cite{sakuma98}. This means, by definition, that there exists a $1$-cell $\tau $ disjoint from $ \mf{k}$ such that
  $ \IS^3 \setminus (\mf{k} \union \tau) $ is an unknotted genus two handlebody $\mc{H} $. It is classical that $ \pi_1(\IS^3 \setminus \mf{k} $) admits a presentation with two generators and one
  relator. There exist two meridian loops of $ \mf{k} $, say $\alpha $  and $\beta$, and a meridian loop of $ \tau $, say $\gamma$, such that the three loops are (isotopic to) simple closed curves
  on $ \Sigma = \partial \mc{H} $, and if $ \IS^3 \setminus \mf{k} $ is hyperbolic then the holonomy representation $ \rho : \pi_1(\IS^3 \setminus \mf{k}) \to \PSL(2,\IC) $ lies in
  $ \Par(\alpha,\beta,\gamma) $; here both $ \rho(\alpha) $, $ \rho(\beta) $ are parabolic but $\rho(\gamma) = 1$. This point of view is used implicitly in work of
  Morimoto, Sakuma, and Yokota~\cite{morimoto96} and of Cho and McCullough~\cite{cho09} (among others) on the classification of unknotting tunnels.
  This note arose partly as a byproduct of a search for effective tools to study unknotting tunnels from the point of view of Kleinian groups, with the goal of recovering information
  like that obtained by Cooper, Futer, and Purcell~\cite{cooper13} using Dehn filling techniques.
\end{cons}

With these constructions in hand, we consider a couple of examples that illustrate \zcref{prob:general}.

\begin{figure}\vspace{1em}
  \labellist
  \small\hair 2pt
  \pinlabel {$X^{-1}Y$} [b] at 498 165
  \pinlabel {link} [b] at 440 165
  \pinlabel {$X^{-2}Y^{-1}X^2Y$} [b] at 329 165
  \pinlabel {$Y$} [b] at 565 165
  \endlabellist
  \centering
  \includegraphics[width=\textwidth]{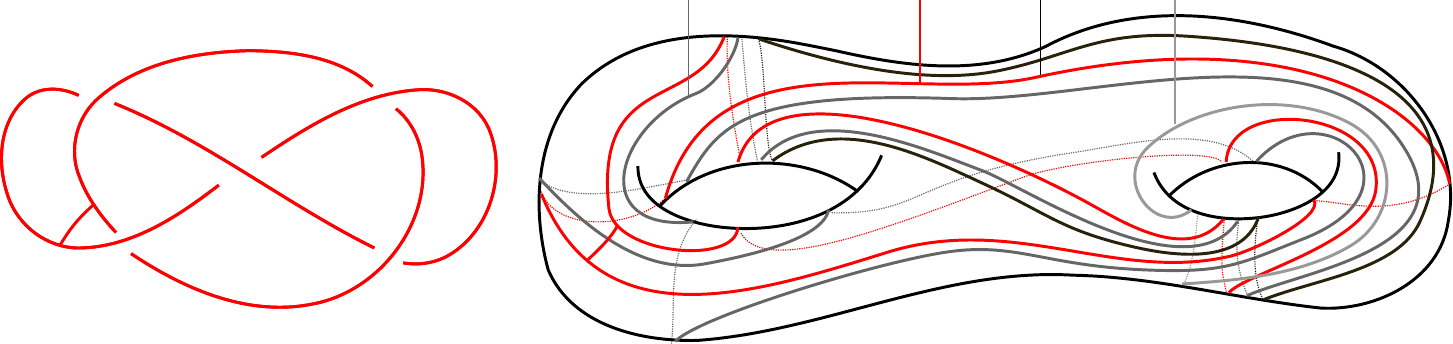}
  \caption{A simple embedding of the $\theta$-graph consisting of the Whitehead link and its unknotting tunnel.\label{fig:whitehead_short}}
\end{figure}

\begin{figure}
  \centering
  \includegraphics[width=.33\textwidth]{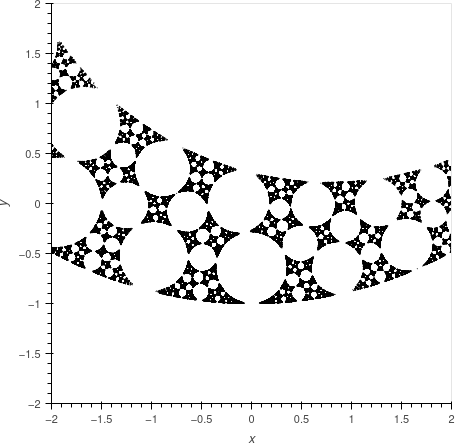}
  \caption{Limit set of the representation of \zcref{ex:whitehead_short}.\label{fig:whitehead_short_limit}}
\end{figure}
\begin{ex}\label{ex:whitehead_short}
  In \zcref{fig:whitehead_short}, we show the $\theta$-graph obtained by adjoining an unknotting tunnel to the Whitehead link, and an embedding of this graph onto the
  genus $2$ surface $\Sigma$. There is a system of three simple closed mutually disjoint curves on $ \Sigma $ so that each curve intersects exactly one arc of the $\theta$-graph;
  we refer to this as a \df{dual system of curves} to the $\theta$-graph. If $ \mc{H} $ is the exterior of $ \Sigma $ in $ \IS^3 $ and has handle core loops represented
  by $ X, Y \in \pi_1(\mc{H}) $, then the three curves have words
  \begin{displaymath}
    X^{-1} Y,\quad X^{-2}Y^{-1}X^2Y,\quad\text{and}\quad Y.
  \end{displaymath}
  There is exactly one representation $\rho$ of $ \langle X,Y \rangle $ (up to trivial symmetries like swapping $ \rho(X) $ with $ \rho(Y) $)
  where all these words have $\tr^2 = 4 $ and which does not have limit set contained in a circle. The plot of its limit set (\zcref{fig:whitehead_short_limit})
  confirms that this is a maximal cusp group.
\end{ex}

\begin{figure}\vspace{1em}
  \labellist
  \small\hair 2pt
  \pinlabel {$X$} [r] at 0 86
  \pinlabel {$Y^{-1}XYX^{-1}YXY^{-1}X^{-1}YX^{-1}Y^{-1}XY^{-1}X^{-1}YX$} [b] at 166 138
  \pinlabel {$Y$} [l] at 411 51
  \endlabellist
  \centering
  \includegraphics[width=.6\textwidth]{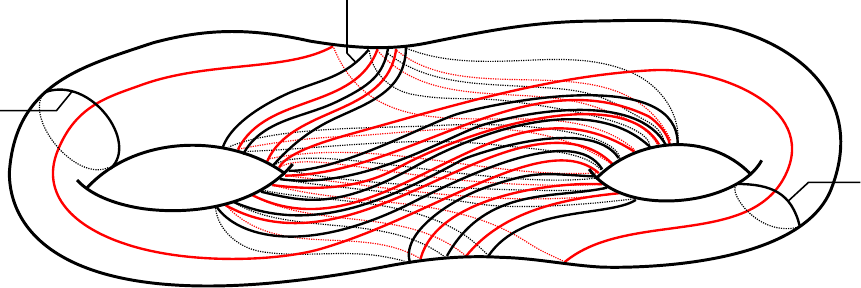}
  \caption{A braid embedding of the Whitehead link and a system of three dual curves. The central black curve is the $3/8$ curve on the four-holed sphere $ \Sigma \setminus \{X \cup Y\} $:
  it wraps $3$ times in the vertical direction and $8$ times horizontally.\label{fig:whitehead_farey}}
\end{figure}

\begin{figure}
  \centering
  \includegraphics[width=.32\textwidth]{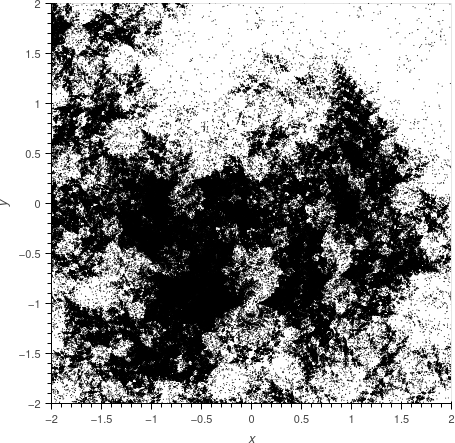}\hfill%
  \includegraphics[width=.32\textwidth]{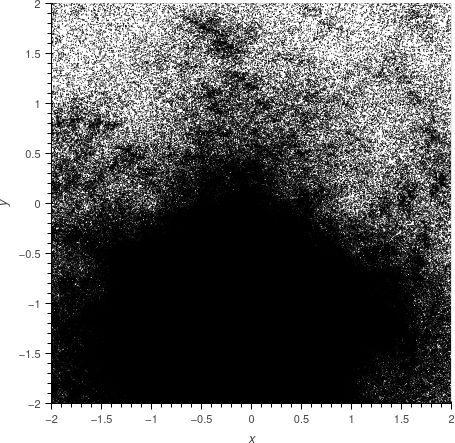}\hfill%
  \includegraphics[width=.32\textwidth]{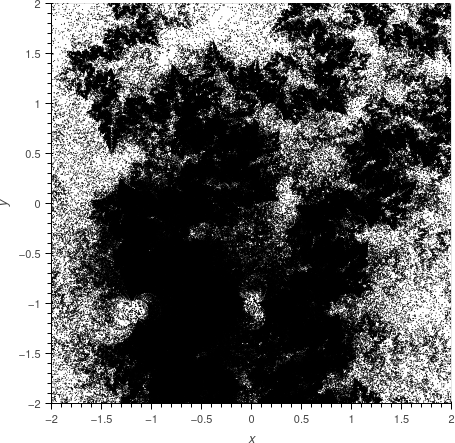}\\
  \includegraphics[width=.32\textwidth]{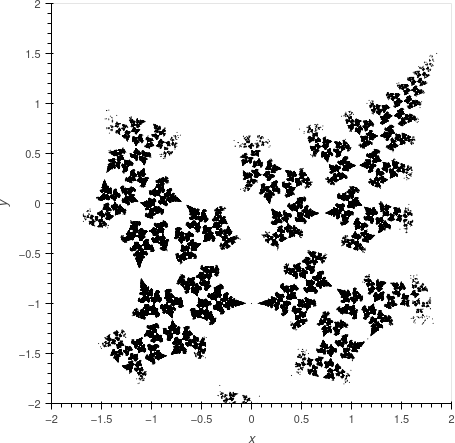}\hspace{.2cm}%
  \includegraphics[width=.32\textwidth]{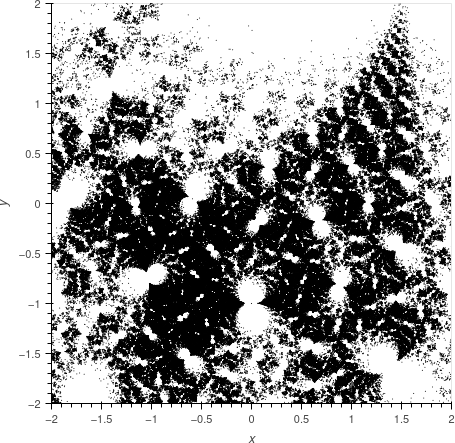}
  \caption{Limit sets of the representations of \zcref{ex:whitehead_farey}.\label{fig:whitehead_farey_limit}}
\end{figure}

\begin{ex}\label{ex:whitehead_farey}
  In \zcref{fig:whitehead_farey}, we consider a different embedding of the Whitehead link onto $ \Sigma $, as the curve of slope $3/8$ (c.f.~\cite[Definition 4.1]{chesebro25}). Again
  we can draw a dual curve system, this time with words
  \begin{displaymath}
    X,\quad Y^{-1}XYX^{-1}YXY^{-1}X^{-1}YX^{-1}Y^{-1}XY^{-1}X^{-1}YX,\quad\text{and}\quad Y.
  \end{displaymath}
  Now there are five different representations that do not have limit set contained in a circle, and we
  show their limit sets in \zcref{fig:whitehead_farey_limit}. The four on the top row appear to be
  indiscrete. On the other hand, the two representations with limit sets plotted in the bottom row
  are discrete: on the left is the maximal cusp, and on the right is the holonomy representation of
  the hyperbolic structure on the Whitehead link complement.
\end{ex}

\subsection{Summary of results}\label{sec:results}
The work in the case of two parabolic generators can be interpreted as giving a complete solution to \zcref{prob:general} in the case that the two words $ \alpha $ and $ \beta $
are the handle cores $X$ and $Y$ of $ \mc{H} $. In this case, the only discrete, non-elementary, non-Fuchsian representations in $\Par(\alpha,\beta,\gamma)$ are:
\begin{enumerate}
  \item[(P1)] Infinite covolume: the maximal cusp group on the boundary of Schottky space corresponding to pinching all three curves to parabolics, and
  \item[(P2)] Finite covolume: the holonomy group of $ \IS^3\setminus\mf{b} $ for a two-bridge link $ \mf{b} \subset \IS^3 $ obtained by gluing a $2$-handle
              to $ \mc{H} $ along $\gamma$, or the holonomy group of a quotient of $ \IS^3\setminus\mf{b} $ by some subgroup of $ \Aut(\mf{b}) $.
\end{enumerate}

Here, we consider the discrete, non-elementary, non-Fuchsian representations $ \rho \in \Par(\alpha,\beta,\gamma)$ such that $ \rho(\alpha) $, $ \rho(\beta) $, and $ \rho(\gamma) $
are all parabolic. We first show that the expected generalisation of (P1) holds in the full rank $2$ setting.

\begin{named}{\zcref{thm:three_parabolics}}
  Suppose that $ \rho \in \Par(\alpha,\beta,\gamma) $ is discrete, non-elementary, and non-Fuchsian, and let $ \Gamma = \rho(\pi_1(\mc{H})) $. If the domain of
  discontinuity $ \Omega(\Gamma) \subset \IP^1(\IC) $ is non-empty, and all of $ \rho(\alpha) $, $ \rho(\beta)$, and $\rho(\gamma) $ are parabolic, then $ \Gamma $
  is a maximally cusped group on the boundary of genus $2$ Schottky space.
\end{named}

Experiments with simple examples, like those above in \zcref{sec:motivation}, might suggest that the parabolic locus of a triplet of curves always consists of exactly one maximal cusp
group up to conjugacy, and at most one additional family of finite covolume groups that come from the finite quotients of a tunnel number $1$ link (generalising the two-bridge link case). In particular,
one might conjecture that every finite covolume group is the holonomy group of an orbifold obtained by gluing a $2$-handle onto the genus $2$ surface along one
of the three curves (so that the curve will lie in the kernel of the representation) and then taking a quotient by the automorphism group of the result, as in (P2) of the parabolic
setting. Indeed, our conversations with other researchers suggest that the following disproof of this conjecture is genuinely surprising.

\begin{named}{\zcref{thm:85link_identification}}
  There exists a triplet of curves $ \{\alpha,\beta,\gamma\} $ on the genus two surface such that $ \Par(\alpha,\beta,\gamma) $ contains a finite covolume representation
  which sends all of $ \alpha $, $ \beta $, and $ \gamma $ to parabolic elements and which is not virtually a tunnel number one link group.
\end{named}

Actually, our main example is a $3$-fold quotient of the Borromean rings, and is striking because the system of curves involved
has no order three symmetry---the system arises from an embedding of the $ 8_5 $ knot onto the genus two surface and has symmetry group $ \IZ/2\IZ \oplus \IZ/2\IZ $.

At the end of the paper we also include an example (\zcref{thm:1046_group}) to show that there can exist discrete groups which lie in intersections of the parabolic loci
corresponding to distinct triplets of curves. This is also a phenomenon that cannot occur in the setting of groups generated by two parabolics (the extensions of the Keen--Series
pleating rays do not collide) but unlike \zcref{thm:85link_identification} it was expected that this could occur, even though no example seems to appear in the literature.

\section{Infinite covolume groups are maximal cusps}
We will require three useful topological facts. We first introduce all the language which will appear in the statements of these facts.
\begin{notation}\label{not:hypotheses}
  Let $ \Gamma $ be a finitely generated non-elementary Kleinian group, let $ M = \IH^3/\Gamma $,
  and let $ N $ be the manifold obtained by removing regular neighbourhoods of every orbifold singularity from $M$. Let $ K $ be an irreducible
  core for $ N $, so in particular $ K $ is a compact $3$-manifold with boundary and $K \hookrightarrow N $ is a homotopy equivalence. Let $ \tau $
  be the number of cusp cylinders, cusp tubes, and orbifold singularities (both ideal arcs and loops) in $ \IH^3/\Gamma $. Finally, if $G$ is a
  group we write $ r(G) $ for the rank of $ G $, and $ \ab(G) $ for the abelianisation $ G/[G,G] $.
\end{notation}

The germinal version of the following result was due to Bers~\cite{bers67}.
\begin{lem}\label{thm:bers_area_formula_1}
  With the setup of \zcref{not:hypotheses},
  $ -\chi(\partial K) \leq 2(r(\Gamma)-1) $, with a strict inequality if $ K$ is not a handlebody.
\end{lem}
\begin{proof}
  Our proof is adapted from Lemma~4.7 of Matsuzaki and Taniguchi~\cite{matsuzaki_taniguchi}, who deal with the torsion-free case. Let $ DK $ be
  the double of $K $ across its boundary, so $ DK $ is a closed $3$-manifold. Then $ 0 = \chi(DK) = 2\chi(K) - \chi(\partial K) $.
  Let $ \beta_i $ denote the $i$th Betti number of $ K $, so $ \chi(K) = 1 - \beta_1 + \beta_2 - 0 $ and $ \beta_1(K) = r(\ab (\pi_1(N)))$. We therefore have
  \begin{displaymath}
    -\chi(\partial K) = -2\chi(K) = 2\beta_1 - 2 - 2\beta_2 \leq 2\beta_1 - 2 =  2(r(\ab(\pi_1(N))) - 1)
  \end{displaymath}
  Here $ \beta_2 = 0 $ if and only if $K$ is a handlebody.

  Now we claim that $ r( \ab(\pi_1(N))) \leq r(\pi_1(M)) $, which will complete the proof. First note that $ M $ is obtained from $ N $ by gluing in
  finitely many singular discs and tori, since a finitely generated Kleinian group has only finitely many distinct conjugacy classes of maximal elliptic subgroups:
  indeed, a finitely generated orbifold group has at most finitely many closed singular loops, and by a version of Ahlfors' finiteness theorem for torsion
  groups~\cite[Theorem~4.1$'$]{matsuzaki_taniguchi} there are also only finitely many ideal singular arcs landing on conformal ends.

  By the Seifert--van Kampen theorem for orbifolds~\cite[{III.$\mc{G}$, 3.10(4)}]{bridson_haefliger}, this implies that $ \pi_1(M) $ is obtained from $ \pi_1(N) $ by a finite sequence of quotients by relations
  of the form $ W_i^{k_i} = 1 $. Passing to the abelianisations, we find that $ \ab(\pi_1(M)) $ is a quotient of $ \ab(\pi_1(N)) $
  by such cyclic groups. By the classification of finitely generated abelian groups, since $  \ab(\pi_1(N)) $ does not contain any finite order
  elements it is of the form $ \IZ^\ell $ for some $ \ell $. A quotient of $ \IZ^\ell $ has a lower rank if and only if the quotient is by one of the factor groups.
  This will happen if and only if one of the words $ W_i^{k_i} $ can be extended to a homology basis of $ N $; but this is not possible since $ W_i^{k_i} $ is imprimitive
  when $ k_i > 1 $. Thus
  \begin{equation}\label{eq:first_betti}
    r(\ab(\pi_1(N))) = r(\ab(\pi_1(M))) \leq r(\pi_1(M))
  \end{equation}
  as desired.
\end{proof}

\begin{lem}\label{thm:bers_area_formula_2}
  With the setup of \zcref{not:hypotheses}, $ \tau \leq 3(r(\Gamma) - 1) $.
\end{lem}
\begin{proof}
  Applying Theorem~4.9 of Matsuzaki and Taniguchi~\cite{matsuzaki_taniguchi} to $N$ and $K$,
  \begin{align*}
    \tau &\leq 3(-1 + \beta_1(K) - \beta_2(K)) + 2\beta_2(K)\\
         &= 3(\beta_1(K) - 1) - \beta_2(K)\\
         &\leq 3(\beta_1(K) - 1).
  \end{align*}
  The result now follows from \eqref{eq:first_betti}.
\end{proof}

\begin{lem}\label{lem:rank_bound}
  With the setup of \zcref{not:hypotheses}, if $n$ is the number of ends of $M$ including singular loops, then $ r(\Gamma) \geq n $.
\end{lem}
\begin{proof}
  We apply the `half lives, half dies' theorem\footnote{I thank Moishe Kohan for sketching this result on Mathematics
  StackExchange, \url{https://math.stackexchange.com/q/5126504/736021}.}~\cite[Lemma~3.5]{hatcher3M} to $K$ to find that in the exact sequence
  \begin{displaymath}\begin{tikzcd}
    H_2(K, \partial K) \arrow[r, "\partial"] & H_1(\partial K) \arrow[r, "i"] & H_1(K)
  \end{tikzcd}\end{displaymath}
  the rank of $ \operatorname{im}(\partial) $ is half of $ r(H_1(\partial K)) $; i.e.\ $ r(\ker(i)) = r(H_1(\partial K))/2 $
  and thus $ r(H_1(K)) \geq r(H_1(\partial K))/2 $. In addition, $ r(H_1(\partial K)) \geq 2n $ since every boundary component of $K$
  has genus at least $1$. Combining this with \eqref{eq:first_betti} we find that
  \begin{displaymath}
   n \leq r(H_1(K)) = r(\ab(\pi_1(M))) \leq r(\pi_1(M))
  \end{displaymath}
  as desired.
\end{proof}

We may prove our first main theorem:
\begin{thm}\label{thm:three_parabolics}
  Suppose that $ \rho \in \Par(\alpha,\beta,\gamma) $ is discrete, non-elementary, non-Fuchsian, and let $ \Gamma = \rho(\pi_1(\mc{H})) $. If $ \Omega(\Gamma) \neq \emptyset $, and
  all of $ \rho(\alpha) $, $ \rho(\beta)$, and $\rho(\gamma) $ are parabolic, then $ \Gamma $ is a maximally cusped group on the boundary of genus $2$ Schottky space.
\end{thm}
\begin{proof}
  We will split into cases. Set $ M  = \IH^3/\Gamma $.
  \begin{enumerate}
    \item \textit{$\Gamma$ has a rank $2$ parabolic subgroup.} Since $ \Omega(\Gamma) \neq \emptyset $, $ K $ has at least $2$ ends so is not a handlebody.
          By \zcref{lem:rank_bound} we see that it has exactly $2$ ends. Let $g$ be the genus of the end which is not the known torus end. Then by
          \zcref{thm:bers_area_formula_1} we know $ -\chi(\partial M) = 2g-2 < 2  $. This means $ g = 0 $ or $ g = 1 $. Both of these situations would contradict $ \Omega(\Gamma) \neq \emptyset $.
    \item We can therefore assume that $ \Gamma $ has only rank $1$ parabolic subgroups and split into two subcases depending on whether or not the three parabolics $ \rho(\alpha) $, $ \rho(\beta)$,
          and $\rho(\gamma) $ lie in $3$ distinct maximal parabolic subgroups or not.
      \begin{enumerate}
        \item \textit{The three parabolics lie in distinct subgroups.} By \zcref{thm:bers_area_formula_2}, there is no elliptic locus. Hence we are in the setting of Theorem~III of
              Keen, Maskit, and Series~\cite{keen91} and $ \Gamma $ is a maximally cusped group.
        \item \textit{The three parabolics do not lie in distinct subgroups.}
              Without loss of generality, suppose that $ \rho(\alpha) $ and $ \rho(\beta) $ lie in the same rank $1$ parabolic subgroup, say generated by $ \rho(\eta) $ where $ \eta \in F\{X,Y\} $.
              Since $ \alpha $ and $ \beta $ are simple closed curves, they define primitive elements of $ H_1(\Sigma) $ (i.e. have coprime coordinates in the usual basis).
              Since $ \ab (F\{X,Y\}) = H_1(\mc{H}) $ is a submodule of $H_1(\Sigma)$, the images $ [\alpha] $ and $ [\beta] $ are also primitive in $ H_1(\mc{H}) $. Since they are linearly
              independent, they form a basis. But the map $ \ab (F\{X,Y\}) \to \ab (\Gamma) $ sends these two linearly independent elements into the same rank one subspace $ \IZ [\rho(\eta)] $.
              Hence $ r(\ab \Gamma) < r(\ab F\{X,Y\}) = 2 $, so there is a strict inequality in \zcref{thm:bers_area_formula_1}: $-\chi(\partial K) < 2 $. But
              since $ \Omega(\Gamma) \neq \emptyset $, $ K $ has at least one boundary component of genus at least $2$ and so $ -\chi(\partial K) \geq 2$, giving a contradiction.
              \qedhere
      \end{enumerate}
  \end{enumerate}
\end{proof}

\section{Embeddings of pretzel knots on a genus two surface}\label{sec:pretzel}
We can deduce very little if $ \Omega(\rho(\pi_1(\mc{H}))) $ is empty. In this section we will show
that such representations can occur.

In order to give an explicit example, we will need a coordinate system on free group representations~\cite[Definition~3.2]{elzenaar25h}.
\begin{defn}
  We let $ \mc{X} $ be the character variety of representations $ F\{X,Y\} \to \PSL(2,\IC) $, where $ F \{X,Y\} $ is the
  free group on the symbols $ X $ and $ Y $. Then $ \mc{X} $ is parameterised by trace parameters by taking the matrices
  \begin{displaymath}\renewcommand*{\arraystretch}{1.5}
    X = \begin{bmatrix} \frac{1}{2}(t_X + i t_{XY}) & -\frac{1}{2} (t_X + v) \\ -\frac{1}{2} (t_X - v) & \frac{1}{2}(t_X - i t_{XY}) \end{bmatrix}\;\text{and}\;
    Y = \begin{bmatrix} \frac{t_Y}{2} - i & \frac{t_Y}{2} \\ \frac{t_Y}{2} & \frac{t_Y}{2} + i \end{bmatrix}
  \end{displaymath}
  where $ v $ satisfies $ v^2 = 4 - t_{XY}^2 $. Here $ \tr X = t_X $, $ \tr Y = t_Y $, and $ \tr XY = t_{XY} $.
\end{defn}

\begin{notation}
  For the remainder of this section we will write $ x \coloneq X^{-1} $, $ y \coloneq Y^{-1} $.
\end{notation}

Let $ \mf{P}(a,b,c) $ be the pretzel link with $ a $, $ b $, and $ c $ tangles.
The link has a natural embedding onto the genus $2$ surface $\Sigma$ given by winding around the three `columns'. It is known
that a sufficient condition for $ \mf{P}(a,b,c) $ to be a knot with tunnel number $1$ is that $ b = 2 $ and $ a \equiv c \equiv 1 \pmod{2} $, \cite[Theorem~2.2, condition (2)]{morimoto96}.
The tunnel location is indicated in \zcref{fig:pretzel_braids_1}. Even when $ b > 2 $ we can draw the arc $ \tau $ on the surface and find a system
of three dual curves, though the $\theta$-curve will no longer be isotopic to an unknotted handlebody in $ \IS^3 $.

\begin{figure}
  \labellist
  \small\hair 2pt
  \pinlabel {$a$ half-twists} [r] at 0 155
  \pinlabel {$c$ half-twists} [l] at 366 155
  \pinlabel {$\tau$} [r] at 283 279
  \endlabellist
  \centering
  \includegraphics[width=.33\textwidth]{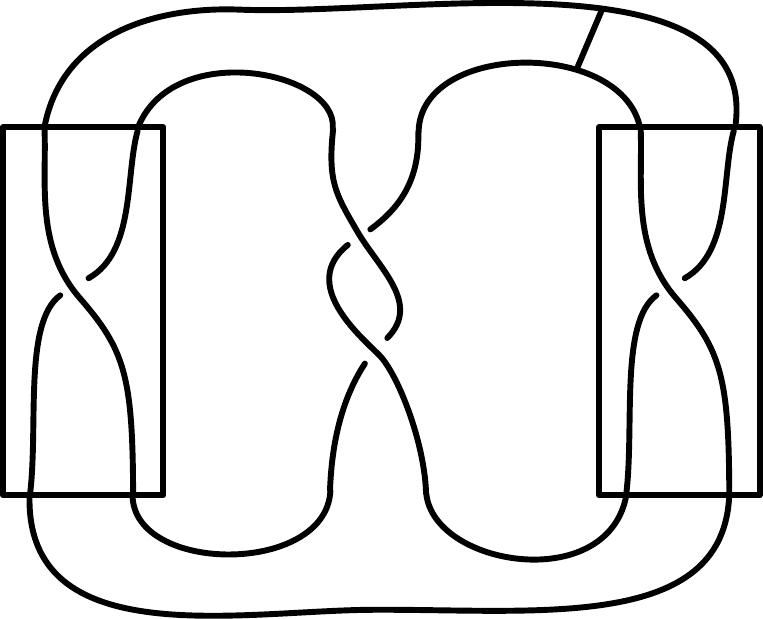}
  \caption{The link $ \mf{P}(a,2,c) $ and the unknotting tunnel $ \tau $.\label{fig:pretzel_braids_1}}
\end{figure}

\begin{lem}\label{lem:pretzel_dual}
  If $ a,c \in \IN $ are odd, and $ b > 2 $ is even, then a system of dual curves to the $ \theta$-graph $ \mf{P}(a,b,c) \union \tau \subset \Sigma $ is given by the words
  \begin{displaymath}
    X^{(a-1)/2} (Xy)^{b/2}, \; X^{(a-1)/2} (Xy)^{b/2} y^{(c-1)/2} (yX)^{b/2}, \; y^{(c-1)/2} (yX)^{b/2}.
  \end{displaymath}
\end{lem}

\begin{figure}
  \centering
  \includegraphics[width=.6\textwidth]{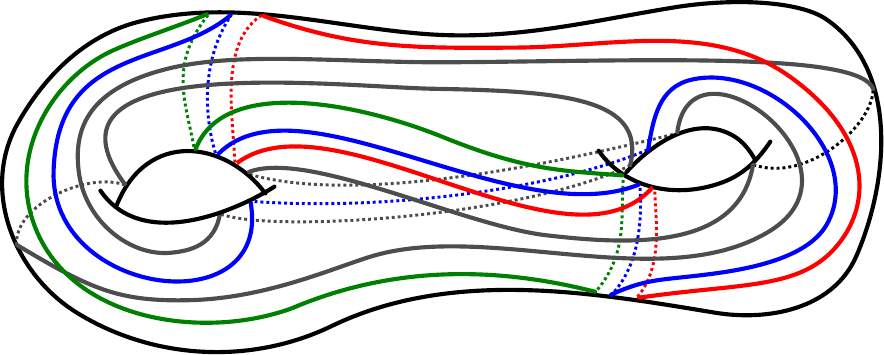}
  \caption{The embedding of $ \mf{P}(1,2,1) $ on $ \Sigma $ (black) and three dual curves.\label{fig:pretzel_braids_2}}
\end{figure}

\begin{proof}
  The embedding is obtained from the system in \zcref{fig:pretzel_braids_2} (where $ a = 1 $, $ b = 2 $, $ c = 1 $) by doing $ (a-1)/2 $, $ b/2-1 $, and $ (c-1)/2 $ Dehn twists
  around the appropriate columns of the surface. The result follows by induction.
\end{proof}

\begin{lem}\label{lem:no_tun_num_1_gps}
  Let $ \alpha, \beta, \gamma $ be the three curves given in \zcref{lem:pretzel_dual} arising from a pretzel link with $ b = 2 $. If $ \rho \in \Par(\alpha,\beta,\gamma) $
  is discrete, non-elementary, and non-Fuchsian, then $ \rho(\alpha) $, $ \rho(\beta) $, and $ \rho(\gamma) $ are all parabolic.
\end{lem}
\begin{proof}
  When $ b = 2 $, the three words are $ X^n y $, $ y^m X $, and $ X^n y^{m+1} X $ where $n = (a-1)/2 $ and $ m = (c-1)/2 $. If $ \rho(X^n y) = 1 $ (resp. $ \rho(y^m X) = 1 $),
  then $ \rho(Y) = \rho(X)^n $ (resp. $ Y^m = X $) and in particular $ X $ and $ Y $ share their entire fixed point sets and so commute~\cite[\S I.D.3]{maskit}. Similarly,
  if $ \rho(X^n y^{m+1} X) = 1 $ then (conjugating by $X$) $ \rho(X)^{n+1} \rho(y)^{m+1}) = 1 $ and again $ X $ and $ Y $ commute. In particular $ \langle X, Y \rangle $
  is elementary.
\end{proof}

\begin{figure}
  \centering
  \includegraphics[width=.32\textwidth]{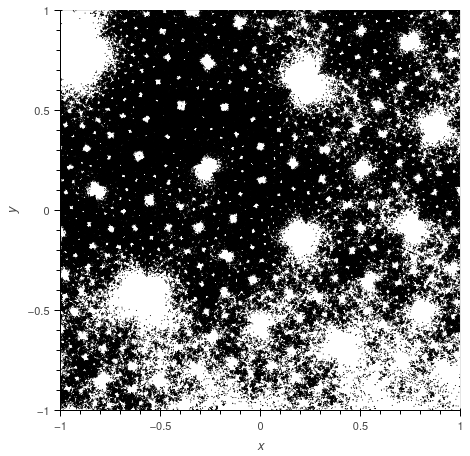}\hspace{.2cm}%
  \includegraphics[width=.32\textwidth]{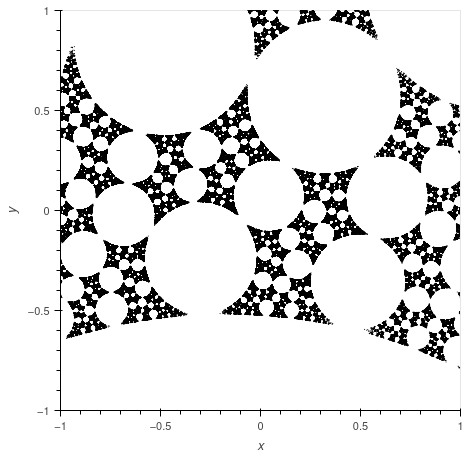}
  \caption{Limit sets of the representations of \zcref{ex:85link}. Left: limit set of \eqref{eq:85finite}. Right: limit set of \eqref{eq:85infinite}.\label{fig:85link_limit}}
\end{figure}

\begin{ex}\label{ex:85link}
  The \href{https://katlas.org/wiki/8_5}{$ 8_5 $ knot} is $ \mf{P}(3,2,3) $, and for this knot the words given by \zcref{lem:pretzel_dual} are $ \alpha = XXy $, $ \beta = yyX $,
  and $ \gamma = XXyyyX $ (compare with \cite[Example~3.7]{elzenaar25h}). Up to conjugacy and permutations of generators, there are exactly two elements of $ \Par(\alpha,\beta,\gamma) $:
  \begin{equation}\label{eq:85finite}
    (t_x, t_y, t_{xy}, v) \approx (-i, -2-i, -2+2i, 2.5440+1.5723i)
  \end{equation}
  and
  \begin{multline}\label{eq:85infinite}
    (t_x, t_y, t_{xy}, v) \approx (-0.7607 - 0.8579i, -0.7607 - 0.8579i,\\ -2.3146 + 2.6103i, 3.0590+1.9751i).
  \end{multline}
  The two limit sets are visible in \zcref{fig:85link_limit}. One immediately conjectures that \eqref{eq:85infinite} is the
  maximal cusp group, and that \eqref{eq:85finite} is a finite covolume discrete group. It is not obvious how to prove this. For instance, \zcref{lem:no_tun_num_1_gps}
  above tells us that the group defined by \eqref{eq:85finite} cannot be the fundamental group of the tunnel number $1$ link group obtained by
  gluing a $2$-handle onto the handlebody along any of the three curves $ \alpha $, $ \beta$, or $ \gamma $. This means that the group cannot arise
  by either of the main constructions of discrete, non-elementary, non-Fuchsian representations in the parabolic locus given in \zcref{sec:motivation}.
\end{ex}

We will confirm in the next theorem that \eqref{eq:85finite} defines a discrete, finite covolume group. This implies that \eqref{eq:85infinite}
is indeed the maximal cusp group associated to the triplet of curves. Rather than giving the proof as directly as possible, we will try to explain
how we went about solving the identification problem in this case as it illustrates how various theoretical results can be used in practice.

\begin{thm}\label{thm:85link_identification}
  The subgroup $G$ of $ \PSL(2,\IC) $ defined by the trace coordinates \eqref{eq:85finite} is discrete and finite covolume.
  Further, $ \IH^3/G$ is an orbifold which is not virtually a tunnel number $1$ link complement.
\end{thm}
\begin{proof}
  Direct computation shows that all three distinguished words are parabolic, and there is an order $3$ elliptic, $ Xy $. We can compute the action of this elliptic by conjugation:
  \begin{displaymath}
    (yX)^{-1} yyX yX = xYyyXyX = xyXyX = x(yX)^2 = x(yX)^{-1} = xxY,
  \end{displaymath}
  which represents the same curve on $ \Sigma $ as $ XXy $. Then
  \begin{displaymath}
    (yX)^{-1} xxY yX = (yX)^2 xxY yX = yXy
  \end{displaymath}
  which represents the same curve as $ yyX $. Finally
  \begin{displaymath}
    (yX)^{-1} yXy yX = yyX.
  \end{displaymath}
  Thus the rank $2$ cusps determined in $\IH^3/G$ by $ XXy $ and $ yyX $ are identical.

  One can check that $ Xy $ commutes with $YYYxxYY $ (which has trace $ -2-3i $), so the subgroup $ \langle Xy, YYYxxYY \rangle $
  represents a singular loop in $ \IH^3/G $ with cone angle $ 2\pi/3 $. Also, $ Xyy $ commutes with $ XXyyyX $, so $ \langle Xyy,XXyyyX\rangle $ is the peripheral
  group of a rank $2$ cusp, the same one represented by $ XXy $ and $ yyX $.

  \begin{figure}
    \centering
    \includegraphics[width=.2\textwidth]{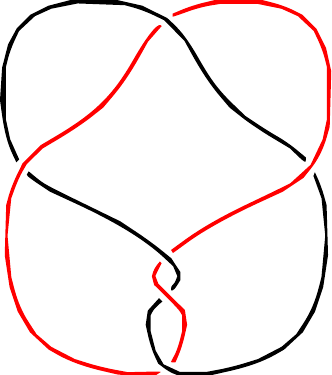}
    \caption{The slope $ 3/10$ $2$-bridge link, Thistlethwaite link \href{https://katlas.org/wiki/L6a2}{$\mathrm{L6a2}$}.\label{fig:85orbifold}}
  \end{figure}

  \begin{figure}
    \centering
    \labellist
    \small\hair 2pt
    \pinlabel {$\infty$} [b] at -2 432
    \pinlabel {$3$} [b] at 110 432
    \pinlabel {unfold to $2\pi/3$} [b] at 444 458
    \pinlabel {fixed axis of order $3$ automorphism} [r] at 245 182
    \endlabellist
    \includegraphics[width=.6\textwidth]{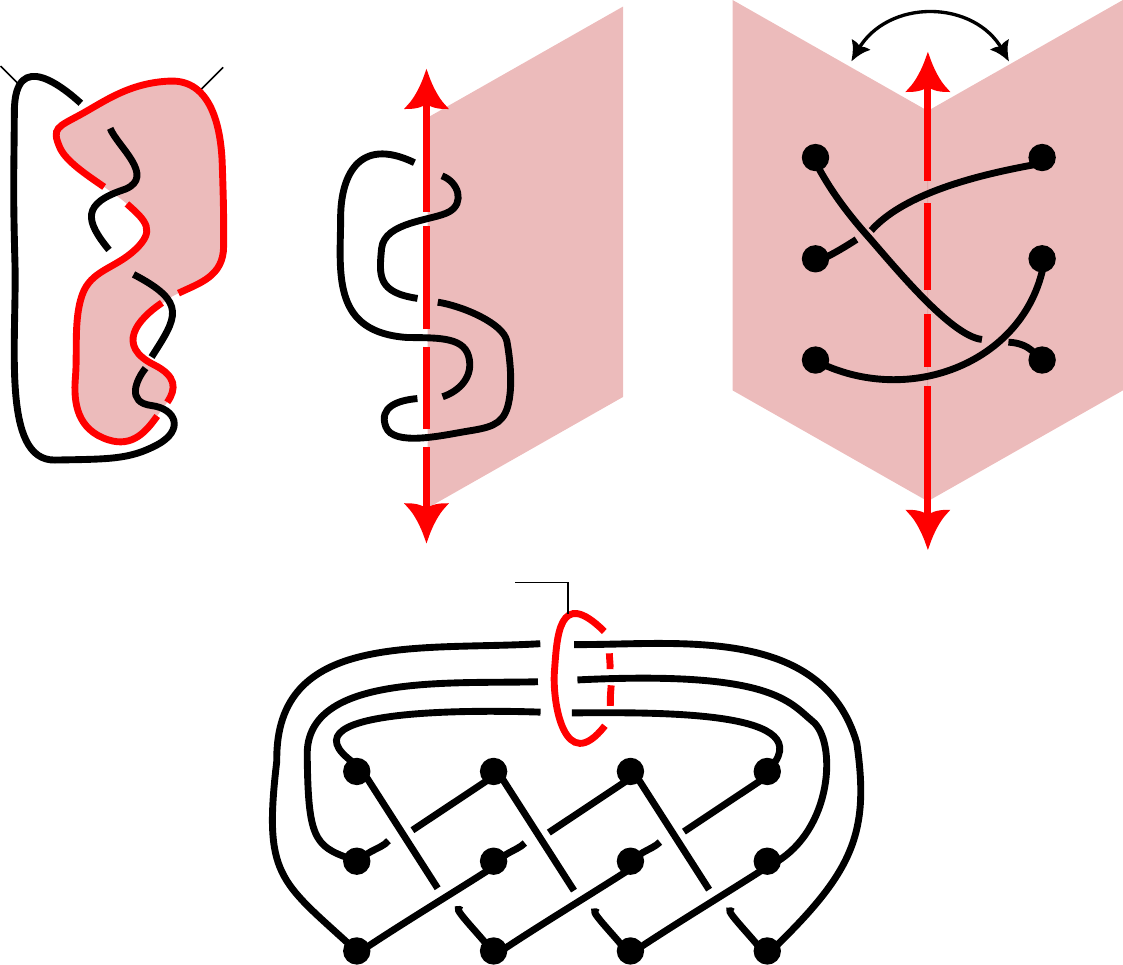}
    \caption{The threefold cover of the $3/10 $ $2$-bridge link orbifold is the Borromean rings.\label{fig:threefoldcover}}
  \end{figure}

  We can therefore guess that $ \IH^3/G $ is a manifold with $2$-generated fundamental group and two singular closed arcs. To confirm this, set $ A = xxY $ and $ B = yX $;
  then $ X = AB $ and $ Y = ABABA $, so $ G = \langle A,B \rangle $. But discrete cofinite groups generated by a pair of elliptics or parabolics have been classified by
  Chesebro, Martin, and Schillewaert~\cite{cms24} so $ G $ must be a Heckoid group, or a 2-bridge link group with one of the components of the link being an order $3$
  cone arc, or the holonomy group of a quotient of such a Heckoid manifold or link orbifold by an automorphism. To check which one, it suffices to iterate through
  the Farey words~\cite[Table~1]{ems22b} until one has trace in $ [-2,2] $. A computer search shows that
  \begin{displaymath}
    W_{3/10} = bAbaBabAbaBaBAbABaBA
  \end{displaymath}
  is the identity in $ G $ and that none of the possible order $2$ elements are order $2$, so $\IH^3/G$ is the orbifold obtained by taking the $2$-bridge link
  of slope $ 3/10 $ and filling one of the components with an order $3$ singular loop, \zcref{fig:85orbifold}. One can verify this by computing the hyperbolic
  structure of this orbifold in \texttt{SnapPy}~\cite{SnapPy} (we have truncated decimals in the output):
  \begin{lstlisting}
    > M=Manifold("L6a2(3,0)")
    > G=M.fundamental_group()
    > G.meridian(0)
    'b'
    > G.meridian(1)
    'bA'
    > G.SL2C('b')
    [-1.0000 - 5.2100e-16*I    0.2588 - 0.9659*I]
    [-0.2588 - 0.9659*I        0.0000           ]
    > G.SL2C('bA')
    [-2.0000 - 1.0000*I    0.7071 + 1.2247*I]
    [-1.2247 - 0.7071*I    4.7200e-16 + 1.0000*I]
    > G.SL2C('bAB')
    [-1.0000 + 1.0000*I    0.7765 - 2.8978*I]
    [-0.9659 + 0.2588*I    1.0000 - 2.0000*I]
  \end{lstlisting}
  Thus the trace-squared of the product of the meridian generators is $ (-i)^2 = -1 $. A direct computation with \eqref{eq:85finite} shows that $ \tr^2 AB = -1 $, and so
  since rank two subgroups of $ \PSL(2,\IC) $ are determined by the set of traces of two generators and their product it follows that the group $G$ is the holonomy group
  of the claimed orbifold.

  One can check directly (\zcref{fig:threefoldcover}) that the threefold manifold cover of the orbifold is the Borromean rings, which is tunnel number $2$ and has rank $3$ fundamental group.
\end{proof}
We remark that in the proof we first use general theory (here, the theory of~\cite{cms24}) to guess what the manifold in question is, but once we have our guess then the verification
is independent of the general theory (we just need to compare the trace parameters of our group to that of the manifold we guess it uniformises)---the identification problem is
one where solutions are hard to find but easy to verify.

Computer experiment shows that finite covolume discrete groups do not appear in the parabolic locus of the curves produced by \zcref{lem:pretzel_dual} very often. However, \zcref{thm:85link_identification}
is not the only example. The following is interesting because it gives a triplet of curves $ \{\alpha,\beta,\gamma\} $ such that there exists some discrete, non-elementary, non-Fuchsian $ \rho \in \Par(\alpha,\beta,\gamma) $
so that all three of $ \rho(\alpha)$, $\rho(\beta)$, and $\rho(\gamma) $ are parabolic, but the image \emph{is} a tunnel number $1$ link group. In this case the Heegaard diagram
defining the link is not $\{\alpha,\beta,\gamma\}$ and so $ \rho $ lies in the intersection of parabolic loci for two genuinely distinct systems of curves. This phenomenon
does not seem to occur in the two-bridge link setting.

\begin{ex}
  The \href{https://katlas.org/wiki/10_46}{$ 10_{46} $ knot} is $ \mf{P}(3,2,5) $, and for this knot the words given by \zcref{lem:pretzel_dual} are $ \alpha = XXy $, $ \beta = yyyX $,
  and $ \gamma = XXyyyyX $. Experimentation suggests that there are exactly two discrete, non-elementary, non-Fuchsian representations in $ \Par(\alpha,\beta,\gamma) $, namely
  \begin{multline}\label{eq:1046finite}
    (t_x, t_y, t_{xy}, v) \approx (-0.5 + 0.8660i, -0.5 - 0.8660i,\\ 2.5 + 0.8660i, 1.2415-1.7439i)
  \end{multline}
  and
  \begin{multline}\label{eq:1046infinite}
    (t_x, t_y, t_{xy}, v) \approx (-0.8414 - 0.7730i, -1.2836 - 0.3640i,\\ -1.5333 + 3.0082i, 3.5231+1.3092i).
  \end{multline}
  We will prove in \zcref{thm:1046_group} below that \eqref{eq:1046finite} is a tunnel number $1$ link group. One can also show that \eqref{eq:1046infinite} is the
  maximal cusp group associated to the system, though this is not trivial---this can be done by constructing a fundamental domain for the action
  of the group on its domain of discontinuity on the Riemann sphere. Beyond these two, there are four other elements of $ \Par(\alpha,\beta,\gamma) $ which all appear
  to be indiscrete. Actually proving this would be a nontrivial task involving iterating through longer and longer words in each group to see if they generate a rank
  two indiscrete subgroup using J\o{}rgensen's inequality~\cite[Theorem~2.17]{matsuzaki_taniguchi}.
\end{ex}

\begin{figure}\vspace{1em}
  \labellist
  \small\hair 2pt
  \pinlabel {$XyyXXXyyXy$} [b] at 230 165
  \pinlabel {$yyXy$} [b] at 318 165
  \pinlabel {$XyyXXX$} [b] at 118 165
  \endlabellist
  \centering
  \includegraphics[width=.7\textwidth]{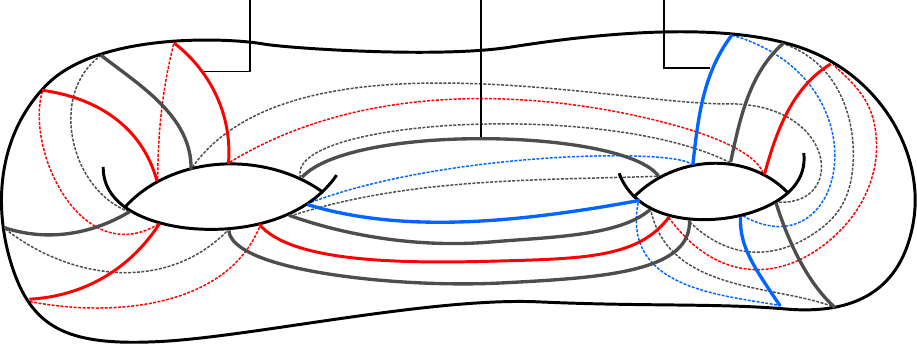}
  \caption{The three new words appearing in the proof of \zcref{thm:1046_group}.\label{fig:m003}}
\end{figure}

\begin{thm}\label{thm:1046_group}
  The group $G$ defined by the trace coordinates \eqref{eq:1046finite} is the holonomy group of a tunnel number $1$ link group.
\end{thm}
\begin{proof}
  One first sees that the relation
  \begin{equation}\label{eq:1046_newrelator}
    (XyyXXX)(yyXy) = 1
  \end{equation}
  holds in $ G $. Here $yyXy$ is a conjugate of $ \beta $, and $XyyXXX$ is a parabolic in $ G $ which is distinct from any of the three parabolics $ \alpha, \beta, \gamma $
  even up to conjugacy and taking inverses. These new parabolic words $ yyXy $ and $ XyyXXX $ and the relator word \eqref{eq:1046_newrelator} represent a triplet of distinct
  simple closed curves on the genus two surface, as shown in \zcref{fig:m003}. Gluing a two-handle to the handlebody $ \mc{H} $ along the word \eqref{eq:1046_newrelator}
  produces a manifold $M$ which is the complement of a tunnel number $1$ link in some closed $3$-manifold, and by the Siefert--van Kampen theorem  we see that
  \begin{equation}\label{eq:1046_pres}
    \pi_1(M) = \langle X, Y : XyyXXXyyXy = 1 \rangle.
  \end{equation}

  We can identify $M$ as the manifold \texttt{m003(0,0)} in the \textsc{SnapPy} census of orientable cusped $3$-manifolds~\cite{SnapPy,callahan99,hildebrand89,burton17}: the
  census manifold has fundamental group $ \langle A, B : abAAbabbb \rangle $ which agrees with \eqref{eq:1046_pres} after the change of variables $ A = X $ and $ B = Yx $
  (so $ X = A $, $ Y = BA $, $ XY = ABA $), so the two are isometric by Mostow--Prasad rigidity~\cite[Theorem~3.32]{matsuzaki_taniguchi}. Further, the holonomy representation of the census manifold
  produced by \textsc{SnapPy} satisfies $ \tr A = -0.5 - 0.8660i $, $ \tr BA = 0.5 - 0.8660i $, and $ \tr ABA = -2.5 + 0.8660i $, agreeing with \eqref{eq:1046finite}. Since rank
  two subgroups of $ \PSL(2,\IC) $ are determined by the set of traces of two generators and their product, this shows that $G$ is conjugate to the census manifold holonomy group,
  completing the proof of the theorem.
\end{proof}

\begin{rem}
  The manifold \texttt{m003(0,0)} is a knot complement in a lens space, see \cite[Example~4.4]{kegel24}.
\end{rem}

\sloppy\printbibliography

\end{document}